\documentclass[12pt]{amsart}
\usepackage[all]{xy}
\usepackage[mathscr]{euscript}
\usepackage{hyperref}
\hypersetup{colorlinks=true,urlcolor=blue,citecolor=blue,linkcolor=blue}
\usepackage{courier}
\usepackage{amsmath,amscd,amssymb,amsthm,latexsym,longtable,diagram, picture}
\usepackage{graphicx}
\usepackage{array}
\usepackage{color}
\usepackage{enumerate}
\usepackage{nicefrac}
\usepackage{listings}
\usepackage{latexsym,bm,bbm,mathrsfs}
\usepackage{hyperref,graphicx, enumerate}
\usepackage{epsfig, xcolor, graphicx}
\usepackage[shortlabels]{enumitem}
\usepackage{indentfirst, setspace}
\usepackage{tikz-cd}

\allowdisplaybreaks
\usepackage{caption}
\usepackage{cancel}
\usepackage[normalem]{ulem}

\usetikzlibrary{calc,matrix,arrows,decorations.markings}

\newcommand{\bbp}{\mathbb{P}}
\newcommand{\bbc}{\mathbb{C}}

\newcommand{\bba}{\mathbb{A}}

\newcommand{\mat}{\begin{pmatrix}}
\newcommand{\emat}{\end{pmatrix}}

\newtheorem{theorem}{Theorem}[section]
\newtheorem{proposition}[theorem]{Proposition}
\newtheorem{corollary}[theorem]{Corollary}
\newtheorem{lemma}[theorem]{Lemma}
\newtheorem{rmk}[theorem]{Remark}

\begin{document}
 \title[Additive Rigidity for Images of Rational Points on Abelian Varieties]
{Additive Rigidity for Images of Rational Points on Abelian Varieties II: The General Case}

\author{Seokhyun Choi}

\address{
Dept. of Mathematical Sciences, KAIST,
291 Daehak-ro, Yuseong-gu,
Daejeon 34141, South Korea
}
\email{sh021217@kaist.ac.kr}

\date{\today}
\subjclass[2020]{Primary 11G05, 11G10}
\keywords{Abelian varieties, Elliptic curves, Additive structures, Mordell-Lang conjecture}

\begin{abstract}
    We study the interaction between the group law on an abelian variety and the additive structure induced on its image under a morphism to a projective space. Let $A/F$ be an abelian variety, $f:A \rightarrow \mathbb{P}^n$ be a morphism which is finite onto its image, and $\Gamma \subseteq A(F)$ be a finite-rank subgroup. We show that for any affine chart $\mathbb{A}^n \subseteq \mathbb{P}^n$ and any finite subset $X \subseteq f(\Gamma) \cap \mathbb{A}^n$, the energy satisfies $E(X) \ll \lvert X \rvert^2$ and the sumset satisfies $\lvert X+X \rvert \gg \lvert X \rvert^2$. Thus images of finite-rank subgroups of abelian varieties cannot have strong additive structure in affine space. This removes the simplicity assumption from the author's previous result. The proof combines the uniform Mordell--Lang conjecture of Gao--Ge--K\"{u}hne with a refined use of the Ueno locus, R\'{e}mond's boundedness theorem for abelian subvarieties of bounded degree, and induction on the dimension of $A$.
\end{abstract}

\maketitle

\section{Introduction}

Many problems in number theory arise from the interaction of distinct algebraic structures. In the setting of abelian varieties, one has the intrinsic group law on the abelian variety, while a morphism to a projective space, after restricting to an affine chart, produces subsets of an affine space with its usual addition. The theme of this paper is that these two additive structures do not interact freely. More precisely, images of finite rank subgroups of abelian varieties under projective morphisms satisfy strong restrictions on their additive structures.

This paper is a continuation of the author's earlier work \cite{Cho26}, where the corresponding additive rigidity statement was proved for simple abelian varieties. The result showed that if $A/F$ is a simple abelian variety, $f:A \rightarrow \bbp^n$ is a morphism which is finite onto its image, and $\Gamma \subseteq A(F)$ is a finite-rank subgroup, then finite subsets $X$ of $f(\Gamma)$ in an affine chart $\bba^n$ satisfy $E(X) \ll \lvert X \rvert^2$ and $\lvert X+X \rvert \gg \lvert X \rvert^2$. The natural question left open by \cite{Cho26} was whether the simplicity assumption on $A$ is actually necessary. The main result of the present paper shows that the simplicity assumption is not necessary: the same additive rigidity phenomenon holds for arbitrary abelian varieties.

A second source of motivation comes from the work of Harrison, Mudgal, and Schmidt \cite{HMS26} on algebraic groups. Their results show that, for correspondences $\mathcal{C}_1,\ldots,\mathcal{C}_g$ between one-dimensional algebraic groups, and for a finite subset $X$ contained in a finite-rank subgroup, $\lvert \mathcal{C}_1(X)+\cdots+\mathcal{C}_g(X) \rvert \gg \lvert X \rvert^g$ under suitable non-degeneracy hypotheses.  From this perspective, our theorem can be viewed as a higher-dimensional analogue in the setting of abelian varieties, with the two-fold sumset $X+X$.

The main new point of this paper is geometric. In the simple case, the structure of abelian subvarieties of $A^2$ is rigid enough to control the exceptional fibers directly. For an arbitrary abelian variety this is no longer true, and many lower-dimensional abelian subvarieties may occur. We overcome this obstruction by using the work of Gao--Ge--Kühne \cite{GGK21} more finely. In addition to the uniform Mordell--Lang bound itself, we use the uniform control of the Ueno locus implicit in their argument. This allows us to separate the non-coset part of each fiber from its positive-dimensional coset contributions. R\'{e}mond's theorem on abelian subvarieties of bounded degree \cite{Rem00} is then used to make the collection of possible abelian subvarieties finite. Together, these inputs replace the simplicity assumption and make an induction on the dimension of $A$ possible.

For a finite subset $X$ of a torsion-free abelian group, the sumset $X+X$ is defined by 
\[X+X := \{a+b \:|\: a,b \in X\},\]
and the energy of $X$ is defined by 
\[E(X) := \lvert \{(a,b,c,d) \in X^4\:|\:a+b=c+d\} \rvert.\]

Our main result is the following.

\begin{theorem}\label{main_theorem}
    Let $A/F$ be an abelian variety of dimension $g$ over an algebraically closed field $F$ of characteristic 0. Let $f:A \rightarrow \bbp^n$ be a morphism which is finite of degree $d$ onto its image, and let $t$ denote the projective degree of $f(A)$ in $\bbp^n$. Let $\Gamma$ be a subgroup of $A(F)$ of finite rank $r$. Then there exists a constant $C(g,d,t)>0$ with the following property.
    
    For every affine chart $\bba^n \subseteq \bbp^n$ and every finite subset $X \subseteq f(\Gamma) \cap \bba^n$, 
    \[E(X) \leq C(g,d,t)^{1+r}\lvert X \rvert^2,\qquad \lvert X+X \rvert \geq \left(C(g,d,t)^{-1}\right)^{1+r}\lvert X \rvert^2.\]
\end{theorem}

\begin{rmk}\label{Cauchy_Schwarz}
    By the Cauchy-Schwarz inequality, 
    \[\lvert X \rvert^4\leq E(X) \lvert X+X \rvert.\]
    Therefore, it suffices to establish the bound 
    \[E(X) \leq C(g,d,t)^{1+r}\lvert X \rvert^2\]
    in Theorem~\ref{main_theorem}.
\end{rmk}

In additive combinatorics, finite sets with small sumset or large energy are regarded as highly additively structured. For instance, Freiman-type theorems show that sets with small doubling are controlled by generalized arithmetic progressions. Theorem~\ref{main_theorem} says that such additive structure cannot occur inside $f(\Gamma) \cap \mathbb A^n$, up to constants depending only on the geometric data and the rank. 

Theorem~\ref{main_theorem} also yields a particularly clean consequence for torsion points, since $\Gamma$ has rank 0 in this case.

\begin{corollary}\label{torsion_corollary}
    Let $A/F$ and $f:A \rightarrow \bbp^n$ be as in Theorem~\ref{main_theorem}. Then for every affine chart $\bba^n \subseteq \bbp^n$ and every finite subset $X \subseteq f(A(F)_{\mathrm{tors}})\cap \bba^n$, 
    \[E(X) \leq C(g,d,t) \lvert X \rvert^2,\qquad \lvert X+X \rvert \geq C(g,d,t)^{-1}\lvert X \rvert^2.\]
\end{corollary}

We now describe the proof. Fix an affine chart $\bba^n \subseteq \bbp^n$, set $V=f^{-1}(\bba^n)$, and consider the morphism
\[\Phi : V \times V \longrightarrow \bba^n,\quad (P,Q) \longmapsto f(P)+f(Q).\]
For each $u \in \bba^n$, we define 
\[Y_u:=\overline{\Phi^{-1}(u)} \subseteq A^2.\]
The energy of a finite set $X \subseteq f(\Gamma) \cap \bba^n$ is then controlled by the number of pairs $(P,Q) \in \Gamma^2$ lying on these fibers $Y_u(F)$. The uniform Mordell--Lang conjecture controls the part of $Y_u$ outside its Ueno locus. The remaining problem is to understand the positive-dimensional cosets contained in the fibers.

The coset contribution splits into two parts. First, if a coset $z+B \subseteq Y_u$ has $\dim B=g$ and meets $V^2$, then the situation is reduced to the same functional equation appearing in \cite[Proposition~5.1]{Cho26},
\[f(a+\alpha P)+f(b+\beta P) \equiv u,\]
and the argument of \cite[Proposition~5.1]{Cho26} shows that such $u$ form a finite exceptional set. Second, if $\dim B < g$, then the two projections of $B$  give two proper abelian subvarieties of $A$. After translating the relevant cosets, the contribution from such $B$ is bounded by the induction hypothesis in smaller dimension. 

The paper is organized as follows. In Section~\ref{Geometric_lemmas}, we collect the geometric lemmas from Gao--Ge--K\"{u}hne and R\'{e}mond. In Section~\ref{Setup}, we set up the fibers $Y_u$, decompose the set $Y_u(F) \cap \Gamma^2$, and prove the finiteness of the exceptional set $\Sigma$. In Section~\ref{Induction}, we prove the induction estimate for lower-dimensional cosets. Finally, in Section~\ref{Proof} we combine these estimates to prove Theorem~\ref{main_theorem}.

\section{Geometric lemmas}\label{Geometric_lemmas}

We collect here the geometric lemmas needed in the proof. The two main inputs are the uniform Mordell--Lang conjecture of Gao--Ge--K\"{u}hne and R\'{e}mond's boundedness theorem for abelian subvarieties of bounded degree.

Let $A/F$ be an abelian variety of dimension $g$ over an algebraically closed field $F$ of characteristic 0. Let $\mathcal{L}$ be an ample line bundle on $A$ and let $X$ be an irreducible closed subvariety of $A$.

The Ueno locus of $X$ is defined to be the union of positive dimensional cosets contained in $X$. By Kawamata \cite[Theorem~4]{Kaw80}, the Ueno locus of $X$ is Zariski closed. Let $X^\circ$ be the complement of the Ueno locus of $X$. Following \cite[Lemma~7.4]{GGK21}, we define $\Sigma(X;A)$ to be the set of positive dimensional abelian subvarieties $B$ satisfying $x+B \subseteq X$ for some $x \in A(F)$, and $B$ is maximal for this property. 

The following lemma is the main form in which the deep results of Gao--Ge--K\"{u}hne enter this paper. It combines the uniform Mordell--Lang estimate in \cite[Theorem~1.1$'$]{GGK21} with the uniform description of the Ueno locus in \cite[Lemma~7.4]{GGK21}.

\begin{lemma}\label{Uniform_Mordell_Lang_conjecture}
    Let $A/F$ be a polarized abelian variety of dimension $g$, $X$ be an irreducible closed subvariety of degree $d$ (with respect to the polarization), and $\Gamma$ be a subgroup of $A(F)$ of finite rank $r$. Then there exist constants $c_i(g,d)>0$ with the following property.
    \begin{enumerate}[(i)]
        \item $\lvert \Sigma(X;A) \rvert \leq c_1(g,d)$.
        \item For each $B \in \Sigma(X;A)$, $\deg_\mathcal{L} B \leq c_2(g,d)$.
        \item For each $B \in \Sigma(X;A)$, the number of cosets $x+B \subseteq X$ which meet $\Gamma$ is at most $c_3(g,d)^{1+r}$.
        \item $\lvert X^\circ(F) \cap \Gamma \rvert \leq c_4(g,d)^{1+r}$.
    \end{enumerate}
\end{lemma}
\begin{proof}
    First, (iv) is exactly \cite[Theorem~1.1$'$]{GGK21}, so we will only prove (i), (ii), and (iii).
    
    Assume first that $X$ generates $A$. Then the assertions are contained in the proof of \cite[Lemma~7.4]{GGK21}. Indeed, by Bogomolov \cite[Theorem~1]{Bog80}, R\'{e}mond \cite[Proposition~4.1]{Rem00}, and \cite[Lemma~2.5]{GGK21}, (i) and (ii) are proved. Fix $B \in \Sigma(X;A)$ and let $n_B$ be the number of cosets $x+B \subseteq X$ which meet $\Gamma$. Gao--Ge--K\"{u}hne constructs a closed subvariety $X_B$ such that $\deg_\mathcal{L} X_B \leq c_5(g,d)$ and $n_B \leq \lvert X_B^\circ(F) \cap \Gamma \rvert$. Then (iv) applied to irreducible components of $X_B$ proves (iii).

    We now prove the general case. As in the end of \cite[Section~7.2]{GGK21}, we let $A'$ be the abelian subvariety of $A$ generated by $X-X$. Then there exists a closed point $Q \in X(F)$  such that $X-Q$ generates $A'$. Note that $\dim A' \leq \dim A=g$ and $\deg_\mathcal{L} (X-Q) = \deg_\mathcal{L} X=d$. Let $\Gamma'$ be the subgroup of $A(F)$ generated by $\Gamma$ and $Q$. Then the rank of $\Gamma' \cap A'(F)$ is $\leq 1+r$. 

    If $x+B \subseteq X$ for an abelian subvariety $B$ of $A$, then $B = (x+B)-(x+B) \subseteq X-X \subseteq A'$, so that $B$ is an abelian subvariety of $A'$. It follows that $\Sigma(X;A)=\Sigma(X-Q;A')$. This proves (i) and (ii) in the general case. Fix $B \in \Sigma(X;A)$ and suppose $x+B \subseteq X$ meets $\Gamma$. Then $(x-Q)+B \subseteq X-Q$ meets $\Gamma' \cap A'(F)$. Thus the number of cosets $x+B \subseteq X$ which meet $\Gamma$ is at most $c_6(g,d)^{2+r}$. This proves (iii) in the general case.
\end{proof}

The second geometric lemma is the boundedness theorem for abelian subvarieties of bounded degree.  We use R\'{e}mond's result \cite[Proposition~4.1]{Rem00} through the following consequence, which allows us later to replace a priori infinite families of abelian subvarieties by a finite set depending only on geometric data.

\begin{lemma}\label{Remond_lemma}
    Let $A/F$ be a polarized abelian variety of dimension $g$ and degree $D$. Then the number of abelian subvarieties $B$ of $A$ satisfying $\deg B \leq d$ is at most $c(g,d,D)>0$.
\end{lemma}
\begin{proof}
    If $F=\bbc$, then the lemma follows from \cite[Proposition~4.1]{Rem00} with the explicit constant $c(g,d,D)$. Let $F$ be an arbitrary algebraically closed field of characteristic 0. Assume there exist $N$ abelian subvarieties $B_1,\ldots,B_N$ of $A$ satisfying $\deg B \leq d$, where $N > c(g,d,D)$. Since $A,\mathcal{L},B_1,\ldots,B_N$ are defined by finitely many algebraic equations, there exists a finitely generated subfield $F_0 \subseteq F$ such that $A,\mathcal{L},B_1,\ldots,B_N$ are all defined over $F_0$. Choose an embedding $F_0 \hookrightarrow \bbc$. Since base change from $F_0$ to $\bbc$ preserves degrees, we get a contradiction.
\end{proof}

\section{Setup for Theorem~\ref{main_theorem}}\label{Setup}

We now set up the geometric framework for the proof of Theorem~1.1. After fixing an affine chart, the additive equation $f(P)+f(Q) = u$ in affine space gives rise to a family of closed subschemes $Y_u$. The additive energy of a finite subset $X$ of $f(\Gamma) \cap \bba^n$ will be estimated by counting the points of $\Gamma^2$ lying on $Y_u(F)$. 

Let $A/F$ be an abelian variety of dimension $g$ over an algebraically closed field $F$ of characteristic 0. Let $f:A \rightarrow \bbp^n$ be a morphism which is finite of degree $d$ onto its image $Z$, and let $t$ denote the projective degree of $Z$ in $\bbp^n$. Let $\Gamma$ be a subgroup of $A(F)$ of finite rank $r$. 

Fix an affine chart $\bba^n \subseteq \bbp^n$ and set $U=Z \cap \bba^n$, $V = f^{-1}(\bba^n) \subseteq A$. Define a morphism 
\[\Phi : V \times V \longrightarrow \bba^n,\quad (P,Q) \longmapsto f(P)+f(Q).\]

For each $u \in \bba^n$, define 
\[Y_u := \overline{\Phi^{-1}(u)} \subseteq A^2,\]
where the closure is the Zariski closure in $A^2$. Then $Y_u$ is a closed subscheme of $A^2$. Let $Y_u^\circ$ be the union of $Y^\circ$ and let $\Omega_u$ be the union of $\Sigma(Y;A^2)$ where $Y$ runs through irreducible components of $Y_u$. 

The polarization of $A^2$ is induced by $f$:
\[\mathcal{L} := (p_1^A)^*\mathcal{L}_0 \otimes (p_2^A)^*\mathcal{L}_0\]
where $p_1^A,p_2^A:A^2 \rightarrow A$ are projections and 
\[\mathcal{L}_0 := f^* \mathcal{O}_{\bbp^n}(1).\]

By the proof of \cite[Proposition~4.1]{Cho26}, $Y_u$ has at most $d^2t^2$ number of irreducible components, and for each irreducible component $Y$ of $Y_u$, the degree of $Y$ with respect to $\mathcal{L}$ is bounded by $2^gd^2t^2$. Now applying Lemma~\ref{Uniform_Mordell_Lang_conjecture} to each irreducible component $Y$ of $Y_u$, we obtain the following proposition.

\begin{proposition}\label{Bounding_proposition}
    There exist constants $c_i(g,d,t)>0$ with the following property.
    \begin{enumerate}[(i)]
        \item $\lvert \Omega_u \rvert \leq c_1(g,d,t)$.
        \item For each $B \in \Omega_u$, $\deg_\mathcal{L} B \leq c_2(g,d,t)$.
        \item For each $B \in \Omega_u$, the number of cosets $z+B \subseteq Y_u$ which meet $\Gamma^2$ is at most $c_3(g,d,t)^{1+r}$.
        \item $\lvert Y_u^\circ(F) \cap \Gamma^2 \rvert \leq c_4(g,d,t)^{1+r}$.
    \end{enumerate}
\end{proposition}

Define $\Omega$ to be the union of $\Omega_u$ over all $u \in \bba^n$. We prove that $\Omega$ is finite. 

\begin{proposition}\label{Omega_finite}
    The set $\Omega$ is finite and 
    \[\lvert \Omega \rvert \leq c_5(g,d,t).\]
\end{proposition}
\begin{proof}
    We first note that 
    \[\deg_\mathcal{L}(A^2) = \binom{2g}{g}d^2t^2.\]
    By Proposition~\ref{Bounding_proposition}~(ii), for every $B \in \Omega$, $\deg_\mathcal{L} B \leq c_2(g,d,t)$. Applying Lemma~\ref{Remond_lemma} proves the proposition.
\end{proof}

To prove Theorem~\ref{main_theorem}, we must analyze the set $(Y_u(F) \cap \Gamma^2) \cap V^2$. First, define $\Omega_u'$ to be the subset of $\Omega_u$ consisting of $B \in \Omega_u$ such that there exists a coset $z+B \subseteq Y_u$ meeting $V^2$. Then we have the following decomposition of $(Y_u(F) \cap \Gamma^2) \cap V^2$:
\begin{equation}\label{decomposition_eq}
    (Y_u(F) \cap \Gamma^2) \cap V^2 \subseteq (Y_u^\circ(F) \cap \Gamma^2) \cup \bigcup_{B \in \Omega_u'} \bigcup_{j=1}^{n_{u,B}} ((z_{u,B,j} + B)(F) \cap \Gamma^2),
\end{equation}
where for each $B \in \Omega_u'$, $n_{u,B} \leq c_3(g,d,t)^{1+r}$.

We further decompose $\Omega_u'$ into two parts. For cosets $z+B \subseteq Y_u$ which meet $V^2$, we have the following lemma.

\begin{lemma}\label{isogeny_lemma}
    Suppose a coset $z+B \subseteq Y_u$ which meets $V^2$ is given. Let $p_1^A,p_2^A:A^2 \rightarrow A$ be projections and let $B_1=p_1^A(B)$, $B_2 = p_2^A(B)$. Then $\left.p_1^A\right|_B:B \rightarrow B_1$, $\left.p_2^A\right|_B:B \rightarrow B_2$ are isogenies and 
    \[\dim B = \dim B_1 = \dim B_2 \leq g.\]
\end{lemma}
\begin{proof}
    Suppose $\left.p_1^A\right|_B$ has positive-dimensional kernel. Then there exists a positive-dimensional abelian subvariety $C$ of $A$ such that $0 \times C \subseteq B$. Take a point $(a,b) \in (z+B) \cap V^2$. Since $0 \times C \subseteq B$, $(a,b+x) \in (z+B) \cap V^2 \subseteq Y_u \cap V^2$ for every $x \in C \cap (V-b)$. Therefore, 
    \[f(a)+f(b+x) \equiv u\]
    for every $x \in C \cap (V-b)$. This shows that $f$ is constant on $(b+C) \cap V$. Since $V$ is Zariski open in $A$, $f$ is constant on $b+C$. Since $C$ is positive-dimensional, this contradicts that $f$ is finite. Hence, $\left.p_1^A\right|_B:B \rightarrow B_1$ is an isogeny. By symmetry, $\left.p_2^A\right|_B:B \rightarrow B_2$ is an isogeny. 
\end{proof}

We conclude that for a coset $z+B \subseteq Y_u$ which meets $V^2$, we have either $\dim B = g$ or $\dim B < g$. Let $\Omega_u^g$ be the set of $B \in \Omega_u'$ satisfying $\dim B=g$, and let $\Omega_u^{<g}$ be the set of $B \in \Omega_u'$ satisfying $\dim B<g$. 

Define $\Sigma$ to be the set of $u \in \bba^n$ such that $\Omega_u^g$ is nonempty. We prove that $\Sigma$ is finite.

\begin{proposition}\label{Sigma_finite}
    The set $\Sigma$ is finite and 
    \[\lvert \Sigma \rvert \leq c_6(g,d,t).\]
\end{proposition}
\begin{proof}
    Let $u \in \Sigma$. Then there exists an abelian variety $B \in \Omega_u^g$ and a coset $z+B \subseteq Y_u$ which meets $V^2$. Let $p_1^A,p_2^A:A^2 \rightarrow A$ be projections and let $B_1=p_1^A(B)$, $B_2 = p_2^A(B)$. By Lemma~\ref{isogeny_lemma}, we have 
    \[\dim B = \dim B_1 = \dim B_2 = g,\]
    and hence $B_1=B_2=A$. We will write $z=(a,b)$, $\alpha = \left.p_1^A\right|_B$, and $\beta = \left.p_2^A\right|_B$. Since a coset $z+B \subseteq Y_u$ meets $V^2$, we obtain the following functional equation:
    \[f(a+\alpha P) + f(b+\beta P) \equiv u.\]
    Now exactly the same proof as in \cite[Proposition~5.1]{Cho26} ends the proof.
\end{proof}

We now pass from the geometry of the fibers $Y_u$ to the combinatorial quantities which appear in the energy calculation. Suppose a finite subset $X \subseteq f(\Gamma) \cap \bba^n$ is given. Let $\mathcal{X} := f^{-1}(X) \cap \Gamma$ be the preimage of $X$ in $\Gamma$. 

For each $u \in \bba^n$, define 
\[N_u := \lvert \{(a,b) \in X^2\:|\:a+b=u\} \rvert\]
and 
\[\mathcal{N}_u := \lvert \{(P,Q) \in \mathcal{X}^2\:|\:f(P)+f(Q)=u\} \rvert.\]
Then we have 
\[N_u \leq \mathcal{N}_u \leq d^2t^2N_u.\]
We also note that the set 
\[\{(P,Q) \in \mathcal{X}^2\:|\:f(P)+f(Q)=u\}.\]
is contained in 
\[(Y_u(F) \cap \Gamma^2) \cap V^2.\]
From the decomposition~\eqref{decomposition_eq} we define 
\[\mathcal{N}_u^\circ := \lvert (Y_u^\circ(F) \cap \Gamma^2) \cap \mathcal{X}^2 \rvert\]
and 
\[\mathcal{N}_u^\bullet := \left\lvert \bigcup_{B \in \Omega_u'} \bigcup_{j=1}^{n_{u,B}} ((z_{u,B,j} + B)(F) \cap \Gamma^2) \cap \mathcal{X}^2 \right\rvert.\]
We further define 
\[\mathcal{N}_{u,B,j} := \lvert ((z_{u,B,j} + B)(F) \cap \Gamma^2) \cap \mathcal{X}^2 \rvert,\quad B \in \Omega_u',1 \leq j \leq n_{u,B}.\]
Then we have 
\[\mathcal{N}_u \leq \mathcal{N}_u^\circ + \mathcal{N}_u^\bullet\]
and 
\[\mathcal{N}_u^\bullet \leq \sum_{B \in \Omega_u'} \sum_{j=1}^{n_{u,B}} \mathcal{N}_{u,B,j}.\]

\section{Induction proposition}\label{Induction}

The purpose of this section is to estimate the contribution of the lower-dimensional cosets in the decomposition of Section~\ref{Setup}. Such cosets project onto proper abelian subvarieties of $A$, so after restricting to the corresponding cosets and translating the base point, the problem reduces to Theorem~1.1 in lower dimension. 

Define $\Omega^{<g}$ to be the union of $\Omega_u^{<g}$ over all $u \in \bba^n$. Fix $B \in \Omega^{<g}$. Let $p_1^A,p_2^A:A^2 \rightarrow A$ be projections and let $B_1=p_1^A(B)$, $B_2 = p_2^A(B)$. By Lemma~\ref{isogeny_lemma}, we have 
\[\dim B = \dim B_1 = \dim B_2 < g,\]
so $B_1,B_2$ are proper abelian subvarieties of $A$. 

\begin{proposition}\label{induction_proposition}
    Assume Theorem~\ref{main_theorem} is true for all dimensions $<g$. Fix $B \in \Omega^{<g}$. Then 
    \begin{equation}\label{induction_sum_estimate}
        \sum_u \sum_{j=1}^{n_{u,B}} \mathcal{N}_{u,B,j}^2 \leq c_7(g,d,t)^{1+r} \lvert X \rvert^2,
    \end{equation}
    where we define $n_{u,B}=0$ if $B \notin \Omega_u^{<g}$.
\end{proposition}
\begin{proof}
    If $R_1 = a+B_1$ and $R_2 = b + B_2$ are cosets of $B_1$ and $B_2$, respectively, then we define $\mathcal{X}_{R_1} := \mathcal{X} \cap R_1$ and $\mathcal{X}_{R_2} := \mathcal{X} \cap R_2$, and then define $X_{R_1} := f(\mathcal{X}_{R_1})$ and $X_{R_2} := f(\mathcal{X}_{R_2})$. 
    
    Recall that 
    \[\mathcal{N}_{u,B,j} := \lvert (z_{u,B,j}+B)(F) \cap \mathcal{X}^2 \rvert.\]
    If we write $z_{u,B,j} = (a_{u,B,j},b_{u,B,j})$, then 
    \[(z_{u,B,j}+B)(F) \cap V^2 \subseteq \{(P,Q) \in (a_{u,B,j}+B_1) \times (b_{u,B,j}+B_2)\:|\:f(P)+f(Q) = u\}.\]
    Thus if we write $R_1 = a_{u,B,j}+B_1$ and $R_2 = b_{u,B,j}+B_2$, then we obtain  
    \[\mathcal{N}_{u,B,j} \leq \lvert \{(P,Q) \in \mathcal{X}_{R_1} \times \mathcal{X}_{R_2}\:|\:f(P)+f(Q)=u\} \rvert.\]
    
    Therefore, the sum in \eqref{induction_sum_estimate} can be estimated by 
    \[\sum_u \sum_{j=1}^{n_{u,B}} \mathcal{N}_{u,B,j}^2 \leq \sum_{R_1 \in A/B_1} \sum_{R_2 \in A/B_2} \sum_{u} \lvert \{(P,Q) \in \mathcal{X}_{R_1} \times \mathcal{X}_{R_2}\:|\:f(P)+f(Q)=u\} \rvert^2.\]
    We now estimate the right sum over $u$.

    Define 
    \[M_{u,R_1,R_2} := \lvert \{(a,b) \in X_{R_1} \times X_{R_2}\:|\:a+b=u\} \rvert\]
    and 
    \[\mathcal{M}_{u,R_1,R_2} := \lvert \{(P,Q) \in \mathcal{X}_{R_1} \times \mathcal{X}_{R_2}\:|\:f(P)+f(Q)=u\} \rvert.\]
    Then we have 
    \[M_{u,R_1,R_2} \leq \mathcal{M}_{u,R_1,R_2} \leq d^2t^2M_{u,R_1,R_2}.\]
    In particular, we obtain 
    \[\sum_u \mathcal{M}_{u,R_1,R_2}^2 \leq d^4t^4 \sum_u M_{u,R_1,R_2}^2.\]
    The right sum is exactly $E(X_{R_1},X_{R_2})$.
    
    In summary, the sum in \eqref{induction_sum_estimate} can be estimated by 
    \begin{equation}\label{induction_energy}
        \sum_u \sum_{j=1}^{n_{u,B}} \mathcal{N}_{u,B,j}^2 \leq d^4t^4 \sum_{R_1 \in A/B_1} \sum_{R_2 \in A/B_2} E(X_{R_1},X_{R_2}).
    \end{equation}
    
    Fix a coset $R_1 \in A/B_1$. Assume $\mathcal{X}_{R_1}$ is nonempty and take $P_{R_1} \in \mathcal{X}_{R_1}$. Then $\mathcal{X}_{R_1} - P_{R_1} \subseteq \Gamma \cap B_1(F)$. Define $f_{R_1}:B_1 \rightarrow \bbp^n$ by 
    \[f_{R_1}(x) = f(P_{R_1}+x).\]
    Then $f_{R_1}$ is a morphism which is finite of degree $d_1 \leq d$ onto its image. Let $\Gamma_{B_1} = \Gamma \cap B_1(F)$, which is of finite rank $\leq r$.
    
    We have to bound the projective degree $t_1$ of $f_{R_1}(B_1)$ in $\bbp^n$. We begin with  
    \[d_1t_1 = \deg_{\mathcal{L}_0} B_1.\]
    Since $\left.p_1\right|_B:B \rightarrow B_1$ is an isogeny, 
    \[\deg_{\mathcal{L}_0} B_1 \leq \deg_{\mathcal{L}} B.\]
    However, by Proposition~\ref{Bounding_proposition}~(ii), $\deg_\mathcal{L} B \leq c_2(g,d,t)$. Therefore, $t_1 \leq c_2(g,d,t)$.
    
    Since $X_{R_1} = f_{R_1}(\mathcal{X}_{R_1} - P_{R_1})$, $X_{R_1}$ is a finite subset contained in $f_{R_1}(\Gamma_{B_1}) \cap \bba^n$. By the induction hypothesis, 
    \[E(X_{R_1}) \leq c_8(g,d,t)^{1+r}\lvert X_{R_1} \rvert^2.\]
    
    By symmetry, given a coset $R_2 = b+B_2$ of $B_2$, if $\mathcal{X}_{R_2}$ is nonempty, we have 
    \[E(X_{R_2}) \leq c_8(g,d,t)^{1+r}\lvert X_{R_2} \rvert^2.\]
    
    Therefore, given cosets $R_1 = a+B_1$ and $R_2 = b+B_2$, we have 
    \begin{equation}\label{energy_bound}
        E(X_{R_1},X_{R_2}) \leq E(X_{R_1})^{1/2}E(X_{R_2})^{1/2} \leq c_8(g,d,t)^{1+r}\lvert X_{R_1} \rvert\lvert X_{R_2} \rvert.
    \end{equation}
    
    Combining \eqref{induction_energy} and \eqref{energy_bound} gives
    \[\sum_u \sum_{j=1}^{n_{u,B}} \mathcal{N}_{u,B,j}^2 \leq c_9(g,d,t)^{1+r} \sum_{R_1 \in A/B_1} \sum_{R_2 \in A/B_2} \lvert X_{R_1} \rvert\lvert X_{R_2} \rvert.\]
    Since $\lvert X_{R_1} \rvert \leq \lvert \mathcal{X}_{R_1} \rvert$ and $\lvert X_{R_2} \rvert \leq \lvert \mathcal{X}_{R_2} \rvert$, 
    \[\sum_u \sum_{j=1}^{n_{u,B}} \mathcal{N}_{u,B,j}^2 \leq c_9(g,d,t)^{1+r} \sum_{R_1 \in A/B_1} \sum_{R_2 \in A/B_2} \lvert \mathcal{X}_{R_1} \rvert\lvert \mathcal{X}_{R_2} \rvert.\]
    However, the right sum is just $\lvert \mathcal{X} \rvert^2$. Since $\lvert \mathcal{X} \rvert \leq dt\lvert X \rvert$, this ends the proof.
\end{proof}

\section{Proof of Theorem~\ref{main_theorem}}\label{Proof}

We now combine the estimates from the previous sections and prove Theorem~\ref{main_theorem} by induction on $g$.

\begin{theorem}\label{main_theorem_re}
    Let $A/F$ be an abelian variety of dimension $g$ over an algebraically closed field $F$ of characteristic 0. Let $f:A \rightarrow \bbp^n$ be a morphism which is finite of degree $d$ onto its image, and let $t$ denote the projective degree of $f(A)$ in $\bbp^n$. Let $\Gamma$ be a subgroup of $A(F)$ of finite rank $r$. Then there exists a constant $C(g,d,t)>0$ with the following property.
    
    For every affine chart $\bba^n \subseteq \bbp^n$ and every finite subset $X \subseteq f(\Gamma) \cap \bba^n$, 
    \[E(X) \leq C(g,d,t)^{1+r}\lvert X \rvert^2,\qquad \lvert X+X \rvert \geq \left(C(g,d,t))^{-1}\right)^{1+r}\lvert X \rvert^2.\]
\end{theorem}
\begin{proof}
    We will use induction on $g$. If $g=1$, then $A$ is an elliptic curve. In particular, $A$ is simple and \cite[Theorem~1.1]{Cho26} implies the theorem in the base case. We assume that the theorem is true for all dimensions less than $g$.

    We begin with 
    \[E(X) = \sum_{u \in X+X} N_u^2\]
    where $N_u$ is defined by 
    \[N_u := \lvert \{(a,b) \in X^2\:|\:a+b=u\} \rvert.\]
    
    First, for sum over $u \in \Sigma$, Proposition~\ref{Sigma_finite} gives 
    \begin{equation}\label{thm_eq1}
        \sum_{u \in \Sigma} N_u^2 \leq \sum_{u \in \Sigma} \lvert X \rvert^2 \leq c_6(g,d,t) \cdot \lvert X \rvert^2.
    \end{equation}

    Next, we estimate the sum 
    \[\sum_{u \notin \Sigma} N_u^2.\]
    Using the notations defined in Section~\ref{Setup}, we obtain 
    \begin{equation}\label{thm_eq2}
        \sum_{u \notin \Sigma} N_u^2 \leq \sum_{u \notin \Sigma} \mathcal{N}_u^2 \leq 2\sum_{u \notin \Sigma} (\mathcal{N}_u^\circ)^2 + 2\sum_{u \notin \Sigma} (\mathcal{N}_u^\bullet)^2.
    \end{equation}

    By Proposition~\ref{Bounding_proposition}~(iv), the first sum is estimated by 
    \begin{equation}\label{thm_eq3}
        \sum_{u \notin \Sigma} (\mathcal{N}_u^\circ)^2 \leq c_4(g,d,t)^{1+r} \sum_{u} \mathcal{N}_u = c_4(g,d,t)^{1+r} \lvert \mathcal{X} \rvert^2 \leq c_{10}(g,d,t)^{1+r} \lvert X \rvert^2.
    \end{equation}

    For the second sum, we begin with 
    \[\mathcal{N}_u^\bullet \leq \sum_{B \in \Omega_u'} \sum_{j=1}^{n_{u,B}} \mathcal{N}_{u,B,j}.\]
    By Cauchy-Schwarz inequality with Proposition~\ref{Bounding_proposition}~(i) and (iii), 
    \[(\mathcal{N}_u^\bullet)^2 \leq c_1(g,d,t)c_3(g,d,t)^{1+r} \sum_{B \in \Omega_u'} \sum_{j=1}^{n_{u,B}} \mathcal{N}_{u,B,j}^2.\]
    Therefore, 
    \begin{equation}\label{thm_eq4}
        \sum_{u \notin \Sigma} (\mathcal{N}_u^\bullet)^2 \leq c_{11}(g,d,t)^{1+r} \sum_{u \notin \Sigma} \sum_{B \in \Omega_u'} \sum_{j=1}^{n_{u,B}} \mathcal{N}_{u,B,j}^2.
    \end{equation}
    Since $u \notin \Sigma$, $\Omega_u^g$ is empty and $\Omega_u' = \Omega_u^{<g}$. Recall that the union $\Omega^{<g}$ of all $\Omega_u^{<g}$ is finite by Proposition~\ref{Omega_finite}. Therefore, we may estimate the right sum in \eqref{thm_eq4} by 
    \begin{equation}\label{thm_eq5}
        \sum_{u \notin \Sigma} \sum_{B \in \Omega_u'} \sum_{j=1}^{n_{u,B}} \mathcal{N}_{u,B,j}^2 \leq \sum_{B \in \Omega^{<g}} \sum_u \sum_{j=1}^{n_{u,B}} \mathcal{N}_{u,B,j}^2.
    \end{equation}
    By Proposition~\ref{induction_proposition} and Proposition~\ref{Omega_finite}, the right sum in \eqref{thm_eq5} can be estimated by  
    \begin{equation}\label{thm_eq6}
        \sum_{B \in \Omega^{<g}} \sum_u \sum_{j=1}^{n_{u,B}} \mathcal{N}_{u,B,j}^2 \leq c_5(g,d,t)c_7(g,d,t)^{1+r}\lvert X \rvert^2.
    \end{equation}
    Combining \eqref{thm_eq4}, \eqref{thm_eq5}, and \eqref{thm_eq6} gives 
    \begin{equation}\label{thm_eq7}
        \sum_{u \notin \Sigma} (\mathcal{N}_u^\bullet)^2 \leq c_{12}(g,d,t)^{1+r}\lvert X \rvert^2.
    \end{equation}

    Combining \eqref{thm_eq2}, \eqref{thm_eq3}, and \eqref{thm_eq7} implies 
    \begin{equation}\label{thm_eq8}
        \sum_{u \notin \Sigma} N_u^2 \leq c_{13}(g,d,t)^{1+r}\lvert X \rvert^2.
    \end{equation}

    Now combining \eqref{thm_eq1} and \eqref{thm_eq8} ends the proof for the upper bound for $E(X)$. The lower bound for $\lvert X+X \rvert$ follows by Remark~\ref{Cauchy_Schwarz}.
\end{proof}

\end{document}